\documentclass[a4paper,12pt,reqno]{amsart}
\usepackage{amsfonts}
\usepackage{amsmath}
\usepackage{amssymb}
\usepackage[a4paper]{geometry}
\usepackage{mathrsfs}
\usepackage{hyperref}
\renewcommand\eqref[1]{(\ref{#1})} %Need with hyperref
\numberwithin{equation}{section}
\theoremstyle{plain}
\newtheorem{thm}{Theorem}[section]

\newtheorem{lem}[thm]{Lemma}

\theoremstyle{definition}

\begin{document}

   \title[Functional inequalities for $p$-sub-Laplacian ]
{Some functional inequalities for the fractional $p$-sub-Laplacian}

\author[A. Kassymov]{Aidyn Kassymov}
\address{
  Aidyn Kassymov:
  \endgraf
  \endgraf
  Institute of Mathematics and Mathematical Modeling
  \endgraf
  125 Pushkin str.
  \endgraf
  050010 Almaty
  \endgraf
  Kazakhstan
  \endgraf
  and
  \endgraf
  Al-Farabi Kazakh National University
  \endgraf
   71 Al-Farabi avenue
   \endgraf
   050040 Almaty
   \endgraf
   Kazakhstan
  {\it E-mail address} {\rm kassymov@math.kz}
  }

 \author[D. Suragan]{Durvudkhan Suragan}
\address{
	Durvudkhan Suragan:
	\endgraf
	Department of Mathematics
	\endgraf
	School of Science and Technology, Nazarbayev University
	\endgraf
	53 Kabanbay Batyr Ave, Astana 010000
	\endgraf
	Kazakhstan
	\endgraf
	{\it E-mail address} {\rm durvudkhan.suragan@nu.edu.kz}
}

\thanks{
The authors were supported in parts by the MESRK grant AP05130981 and the target program BR05236656.
}


     \begin{abstract}
 In this paper we study the fractional Dirichlet $p$-sub-Laplacian in a Haar measurable set on homogeneous Lie groups.
 We prove fractional Sobolev and Hardy inequalities and we also present a Lyapunov-type inequality  for the  fractional $p$-sub-Laplacian. As a consequence of the Lyapunov-type inequality we show an estimate of the first eigenvalue in a quasi-ball for the  Dirichlet  fractional $p$-sub-Laplacian.
     \end{abstract}
     \maketitle

\section {Introduction}
\subsection{Homogenous Lie group}
 We recall that a family of dilations of a Lie algebra $\mathfrak{g}$ is a family of linear mappings of the form
\begin{equation}
D_{\lambda}=\text{Exp}(A\ln \lambda)=\sum^{\infty}_{k=0}(\ln (\lambda)A)^{k},
\end{equation}
where $A$ is a diagnalisable linear operator on $\mathfrak{g}$ with positive eigenvalues, and  $D_{\lambda}$ is a morphism of the Lie algebra $\mathfrak{g}$, that is, a linear mapping from $\mathfrak{g}$ to itself which respects to the Lie bracket:
\begin{equation}
\forall X, Y\in \mathfrak{g},\,\,\lambda>0,\,\,[D_{\lambda}X, D_{\lambda}Y]=D_{\lambda}[X,Y].
\end{equation}

A homogeneous group is a connected simply connected Lie group whose Lie algebra is
equipped with dilations. We
denote by
\begin{equation}
Q:=\text{Tr} A,
\end{equation}
the homogeneous dimension of a homogeneous group $\mathbb{G}$.

For any $\mathbb{G}$ there exists  homogeneous quasi-norm, which is a continuous non-negative function
\begin{equation}
\mathbb{G}\ni x\mapsto q(x)\in[0,\infty),
\end{equation}
with the properties

\begin{itemize}
\item[a)] $q(x)=q(x^{-1})$ for all $x\in\mathbb{G}$,
\item[b)] $q(\lambda x)=\lambda q(x)$ for all $x\in \mathbb{G}$ and $\lambda>0$,
\item[c)] $q(x)=0$ iff $x=0$.
\end{itemize}
Note that for the simplicity of notation, here and after we assume that the origin $0$ of $\mathbb{R}^{N}$ is the identity of $\mathbb{G}$. This assumption is not restrictive due to properties of isomorphic Lie groups.
We also use the following well-known polar decomposition on homogeneous Lie groups (see, e.g. \cite[Section 3.1.7]{FR}):
there is a (unique)
positive Borel measure $\sigma$ on the
unit quasi-sphere
\begin{equation}\label{EQ:sphere}
\omega_{Q}:=\{x\in \mathbb{G}:\,q(x)=1\},
\end{equation}
so that for every $f\in L^{1}(\mathbb{G})$ we have
\begin{equation}\label{EQ:polar}
\int_{\mathbb{G}}f(x)dx=\int_{0}^{\infty}
\int_{\omega_{Q}}f(ry)r^{Q-1}d\sigma(y)dr.
\end{equation}
Let $p>1$, $s\in(0,1)$ and $\mathbb{G}$ be a homogeneous Lie group with homogeneous dimension $Q$. For a (Haar) measurable and compactly supported function $u$ the fractional $p$-sub-Laplacian $(-\Delta_{p,q})^{s}$ on $\mathbb{G}$ can be defined as
\begin{equation}
(-\Delta_{p,q})^{s}u(x)=2\lim_{\delta\searrow 0}\int_{\mathbb{G} \setminus B_{q}(x,\delta)}\frac{|u(x)-u(y)|^{p-2}(u(x)-u(y))}{q^{Q+sp}(y^{-1}\circ x)}dy, \,\,\,x\in \mathbb{G},
\end{equation}
where $q$ is a quasi-norm on $\mathbb{G}$ and $B_{q}(x,\delta)$ is a quasi-ball with respect to $q$, with radius $\delta$ centered at $x\in\mathbb{G}$ .

For a measurable function $u:\mathbb{G}\rightarrow \mathbb{R}$ we define the Gagliardo seminorm by
\begin{equation}\label{gsmnm}
[u]_{s,p,q}=\left( \int_{\mathbb{G}} \int_{\mathbb{G}}\frac{|u(x)-u(y)|^{p}}{q^{Q+sp}(y^{-1}\circ x)}dxdy\right)^{1/p}.
\end{equation}
Now we recall the definition of the fractional Sobolev spaces on homogeneous Lie groups denoted
 by $W^{s,p,q}(\mathbb{G})$. For $p\geq1$ and $s\in(0,1)$, the functional space
\begin{equation}
W^{s,p,q}(\mathbb{G})=\{u\in L^{p}(\mathbb{G}): u \text {\,\,is measurable}, [u]_{s,p,q}<+\infty \},
\end{equation}
endowed with the norm
\begin{equation}
\|u\|_{W^{s,p,q}(\mathbb{G})}=(\|u\|_{L^{p}(\mathbb{G})}+[u]_{s,p,q})^{1/p},\,\,\,u\in W^{s,p,q}(\mathbb{G}),
\end{equation}
is called the fractional Sobolev spaces on $\mathbb{G}$.

Similarly, if $\Omega\subset \mathbb{G}$ is a Haar measurable set, we define the Sobolev space $W^{s,p,q}(\Omega)$
\begin{equation}
W^{s,p,q}(\Omega)=\{u\in L^{p}(\Omega): u \text {\,\,is measurable},\left( \int_{\Omega}\int_{\Omega}\frac{|u(x)-u(y)|^{p}}{q^{Q+sp}(y^{-1}\circ x)}dxdy\right)^{1/p}<+\infty\},
\end{equation}
endowed with norm
\begin{equation}
\|u\|_{W^{s,p,q}(\Omega)}=\left(\|u\|_{L^{p}(\Omega)}+\int_{\Omega}\int_{\Omega}\frac{|u(x)-u(y)|^{p}}{q^{Q+sp}(y^{-1}\circ x)}dxdy\right)^{1/p},\,\,\,u\in W^{s,p,q}(\Omega).
\end{equation}
Let define Sobolev space $W^{s,p,q}_{0}(\Omega)$  as the completion of $C^{\infty}_{0}(\Omega)$ with respect to the norm $\|u\|_{W^{s,p,q}(\Omega)}$.
\subsection {Fractional Sobolev inequality}
Let $\Omega\subset\mathbb{R}^{N}$ be a measurable set and $1<p<N$, then the (classical) Sobolev inequality is formulated as
\begin{equation}
\|u \|_{L^{p^*}(\Omega)}\leq C\|\nabla u \|_{L^{p}(\Omega)},\,\,\,u\in C^{\infty}_{0}(\Omega),
\end{equation}
where $C=C(N,p)>0$ is a positive constant, $p^{*}=\frac{Np}{N-p}$ and $ \nabla$ is a standard gradient in $\mathbb{R}^{N}$.

In \cite{DPV} the authors obtained the fractional Sobolev inequality in the case $N>sp$, $1<p<\infty,$ and $s\in(0,1)$, for any measurable and compactly supported function $u$ one has
\begin{equation}
\|u \|_{L^{p^*}(\mathbb{R}^{N})}\leq C [u]_{s,p}^{p},
\end{equation}
where $C=C(N,p,s)>0$ is a suitable constant, $[u]_{s,p}^{p}=\int_{\mathbb{R}^{N}}\int_{\mathbb{R}^{N}}
\frac{|u(x)-u(y)|^{p}}{|x-y|^{N+sp}}dxdy$ and $p^*=\frac{Np}{N-sp}$.
There is a number of generalisations and extensions of above Sobolev's inequality. For example, in \cite{Abd} the
authors proved the following weighted fractional Sobolev inequality:
Let $1<p<\frac{N}{s}$ and $0<\beta<\frac{N-ps}{2}$, then for all $u\in C^{\infty}_{0}(\mathbb{R}^{N})$ one has
\begin{equation}
C\int_{\mathbb{R}^{N}}\int_{\mathbb{R}^{N}}\frac{|u(x)-u(y)|^{p}}{|x-y|^{N+ps}|x|^{\beta}|y|^{\beta}}dxdy\geq \left(\int_{\mathbb{R}^{N}}\frac{|u|^{p^{*}}}{|x|^{\frac{2\beta p^{*}}{p}}}dx\right)^{\frac{p}{p^{*}}},
\end{equation}
where $C=C(N,p,s)>0$ and $p^*=\frac{Np}{N-sp}$.

The Sobolev inequality is one of the most important tools in PDE and variational problems. In this paper one of our aims is to obtain an analogue of the fractional Sobolev inequality on the homogeneous Lie groups. The result is stated in Theorem \ref{sob}.

\subsection {Fractional Hardy inequality}
In PDE the $L^{p}$-Hardy inequality takes form
\begin{equation}
\left\|\frac{u(x)}{|x|}\right\|_{L^{p}(\mathbb{R}^{N})}\leq\frac{p}{N-p}\|\nabla u \|_{L^{p}(\mathbb{R}^{N})}, \,\,1< p<N,
\end{equation}
where $u\in C^{\infty}_{0}(\mathbb{R}^{N})$ and $\nabla$ is the standard gradient in $\mathbb{R}^{N}$. In  \cite{Abd} the authors studied the weighted fractional $p$-Laplacian and established the following weighted fractional $L^{p}$-Hardy inequality
\begin{equation}
C\int_{\mathbb{R}^{N}}\frac{|u(x)|^{p}}{|x|^{ps+2\beta}}dx\leq\int_{\mathbb{R}^{N}}\int_{\mathbb{R}^{N}}\frac{|u(x)-u(y)|^{p}}{|x-y|^{N+ps}|x|^{\beta}|y|^{\beta}}dxdy,
\end{equation}
where $\beta<\frac{N-ps}{2}$, $u\in C_{0}^{\infty}(\mathbb{R}^{N})$ and $C>0$ is a positive constant. For  general Lie group discussions of Hardy type inequalities we refer to recent papers \cite{RSAdv17}, \cite{RSJDE17} and \cite{Ruzhansky-Suragan:squares} as well as references therein.
In the present paper  we established an analogue of the fractional Hardy inequality on the homogeneous Lie groups (see Theorem \ref{harun1}).

\subsection {Lyapunov-type inequality}
Historically, in Lyapunov's work \cite{Lyap} the following one-dimensional homogeneous Dirichlet boundary value problem was studied (for the second order ODE)

\begin{equation}\label{slayp}
 \begin{cases}
   u''(x)+\omega(x)u(x)=0,\,\,x\in(a,b),\\
   u(a)=u(b)=0,
 \end{cases}
\end{equation}
and it was proved that, if $u$ is a non-trivial solution of \eqref{slayp} and $\omega(x)$ is a real-valued and continuous function on $[a,b]$, then
 \begin{equation}\label{slaypin}
\int^{b}_{a}|\omega(x)|dx>\frac{4}{b-a}.
\end{equation}
Inequality \eqref{slaypin} is called a (classical) Lyapunov inequality. Nowadays, there are many extensions of above Lyapunov's inequality. In \cite{Ed} the author obtains the Lyapunov inequality for the one-dimensional Dirichlet $p$-Laplacian
\begin{equation}\label{playp}
 \begin{cases}
   (|u'(x)|^{p-2}u'(x))'+\omega(x)u(x)=0,\,\,x\in(a,b),\,\,\,1<p<\infty,\\
   u(a)=u(b)=0,
 \end{cases}
\end{equation}
where $\omega(x)\in L^{1}(a,b)$, so
 \begin{equation}\label{playpin}
\int^{b}_{a}|\omega(x)|dx>\frac{2^{p}}{(b-a)^{p-1}},\,\,\,1<p<\infty.
\end{equation}
Obviously, taking $p=2$ in \eqref{playpin}, we recover \eqref{slaypin}.

Recently in the paper \cite{Kir} the authors obtained interesting results concerning Lyapunov inequalities for the multi-dimesional fractional $p$-Laplacian $(-\Delta_{p})^{s}$, $1<p<\infty,\,\,s\in(0,1)$, with a homogeneous Dirichlet boundary condition, that is,
\begin{equation}\label{fplayp}
 \begin{cases}
   (-\Delta_{p})^{s}u=\omega(x)|u|^{p-2}u,\,\,x\in\Omega,\\
   u(x)=0,\,\,x\in \mathbb{R}^{N}\setminus\Omega,
 \end{cases}
\end{equation}
where $\Omega\subset \mathbb{R}^{N}$ is a measurable set, $1<p<\infty,$ and $s\in(0,1).$
Let us recall the following result of \cite{Kir}.
\begin{thm}
Let $\omega\in L^{\theta}(\Omega)$ with $N>sp,\,\,\frac{N}{sp}<\theta<\infty,$  be a non-negative weight. Suppose that problem \eqref{fplayp} has a non-trivial weak solution $u\in W^{s,p}_{0}(\Omega)$. Then
\begin{equation}\label{fplyapq>sp}
\left(\int_{\Omega}\omega^{\theta}(x) \,dx \right)^{\frac{1}{\theta}}>\frac{C}{r_{\Omega}^{sp-\frac{N}{\theta}}},
\end{equation}
where $C>0$ is a universal constant and $r_{\Omega}$ is the inner radius of $\Omega$.
\end{thm}

 One of our goals in this paper is to extend Lyapunov-type inequality \eqref{fplyapq>sp} for the fractional $p$-sub-Laplacian with homogeneous Dirichlet boundary condition (in the case $Q>sp$, $1<p<\infty$ and $s\in(0,1)$) on the homogeneous Lie groups. This result is given in Theorem \ref{thmlyap}.

 Summarizing our main results of this paper, we present the following facts:
\begin{itemize}
\item An analogue of the fractional Sobolev inequality on $\mathbb{G}$;
\item An analogue of the fractional Hardy inequality on $\mathbb{G}$;
\item A Lyapunov-type inequality for the fractional Dirichlet $p$-sub-Laplacian on $\mathbb{G}$;
\item  An estimate of the first eigenvalue for the  Dirichlet fractional $p$-sub-Laplacian on $\mathbb{G}$.

\end{itemize}
The paper is organized as follows. In Section \ref{SEC:2} we prove (sub-elliptic) functional inequalities such as analogues of the fractional Sobolev and Hardy inequalities on $\mathbb{G}$. In Section \ref{SEC:3} we prove a Lyapunov-type inequality for the fractional Dirichlet  $p$-sub-Laplacian. Then we also give an application of the Lyapunov-type inequality.

\section{Fractional Sobolev and Hardy inequalities}
\label{SEC:2}
In this section we prove fractional Sobolev and Hardy inequalities on the homogeneous Lie groups.

Let $Q$ be a homogeneous dimension of a homogeneous Lie group $\mathbb{G}$. To prove an analogue of the fractional Sobolev inequality, first we present some preliminary results. Here we follow a similar scheme as in \cite{DPV}, but now on the homogeneous Lie groups.

\begin{lem}\label{lem1}
Let $p>1$, $s\in(0,1)$ and $K\subset \mathbb{G}$ be Haar measurable set. Fix $x\in \mathbb{G}$ and a quasi-norm $q$ on $\mathbb{G}$, then we have
\begin{equation}
\int_{K^{c}}\frac{dy}{q^{Q+sp}(y^{-1}\circ x)}\geq C|K|^{-sp/Q},
\end{equation}
where $C=C(Q,s,p,q)$ is a positive constant, $K^{c}=\mathbb{G}\setminus K$ and $|K|$ is the Haar measure of $K$.
\end{lem}
\begin{proof}
Let $\delta:=\left(\frac{|K|}{\omega_{Q}}\right)^{1/Q},$ where $\omega_{Q}$ is a surface measure of the unit quasi-ball on $\mathbb{G}$. Then, we have
 \begin{equation}\label{3.1.1}
|K^{c}\cap B_{q}(x,\delta)|=|B_{q}(x,\delta)|-|K\cap B_{q}(x,\delta)|=|K|-|K\cap B_{q}(x,\delta)|=|K\cap B_{q}^{c}(x,\delta)|,
\end{equation}
where $|\cdot|$ is the Haar measure on $\mathbb{G}$ and $B_{q}(x,\delta)$ is a quasi-ball centered at $x$ with radius $\delta$.
Then,
\begin{align*}
	\int_{K^{c}}\frac{dy}{q^{Q+sp}(y^{-1}\circ x)}  &=\int_{K^{c}\cap B_{q}(x,\delta)}\frac{dy}{q^{Q+sp}(y^{-1}\circ x)}+\int_{K^{c}\cap B_{q}^{c}(x,\delta)}\frac{dy}{q^{Q+sp}(y^{-1}\circ x)}
	\\  &
	\geq \int_{K^{c}\cap B_{q}(x,\delta)}\frac{dy}{\delta^{Q+sp}}+\int_{K^{c}\cap B_{q}^{c}(x,\delta)}\frac{dy}{q^{Q+sp}(y^{-1}\circ x)}
	\\  &
	=\frac{|K^{c}\cap B_{q}(x,\delta)|}{\delta^{Q+sp}}+\int_{K^{c}\cap B_{q}^{c}(x,\delta)}\frac{dy}{q^{Q+sp}(y^{-1}\circ x)}.
\end{align*}
By using \eqref{3.1.1} we establish
\begin{align*}
\int_{K^{c}}\frac{dy}{q^{Q+sp}(y^{-1}\circ x)} & \geq \frac{|K^{c}\cap B_{q}(x,\delta)|}{\delta^{Q+sp}}+\int_{K^{c}\cap B_{q}^{c}(x,\delta)}\frac{dy}{q^{Q+sp}(y^{-1}\circ x)}
\\ &
= \frac{|K\cap B_{q}^{c}(x,\delta)|}{\delta^{Q+sp}}+\int_{K^{c}\cap B_{q}^{c}(x,\delta)}\frac{dy}{q^{Q+sp}(y^{-1}\circ x)}
\\ &
\geq \int_{K\cap B_{q}^{c}(x,\delta)}\frac{dy}{q^{Q+sp}(y^{-1}\circ x)}+\int_{K^{c}\cap B_{q}^{c}(x,\delta)}\frac{dy}{q^{Q+sp}(y^{-1}\circ x)} \\ & =\int_{B_{q}^{c}(x,\delta)}\frac{dy}{q^{Q+sp}(y^{-1}\circ x)}.
\end{align*}
Now using the polarization formula \eqref{EQ:polar} we obtain that
\begin{equation}
\int_{K^{c}}\frac{dy}{q^{Q+sp}(y^{-1}\circ x)}\geq C|K|^{-sp/Q}.
\end{equation}
\end{proof}
\begin{lem}[\cite{DPV}, Lemma 6.2]\label{sob2}
Fix $T>1$. Let $p>1$ and $s\in(0,1)$ be such that $Q>sp$, $m\in\mathbb{Z}$ and $a_{k}$ be a bounded, decreasing, nonnegative sequence with $a_{k}=0$ for any $k\geq m$. Then
\begin{equation*}
\sum_{k\in\mathbb{Z}}a_{k}^{(Q-sp)/Q}T^{k}\leq C\sum_{k\in\mathbb{Z},\,{a_{k}\neq0}}a_{k+1}a_{k}^{-sp/Q}T^{k},
\end{equation*}
for a positive constant $C=C(Q,s,p,T)>0$.
\end{lem}

\begin{lem}\label{3.3}
Let $p>1$, $s\in(0,1),$ $Q>sp$ and $q$ be a quasi-norm on $\mathbb{G}$. Let $u\in L^{\infty}(\mathbb{G})$ be compactly supported and  $a_{k}:=|\{|u|>2^{k}\}|$ for any $k\in\mathbb{Z}$. Then,
\begin{equation}
C\sum_{k\in\mathbb{Z},\,{a_{k}\neq0}}a_{k+1}a^{-sp/Q}_{k}2^{kp}\leq[u]^{p}_{s,p,q},
\end{equation}
where $C=C(Q,p,s,q)$ is a positive constant and $[u]_{s,p,q}$ is defined by \eqref{gsmnm}.
\end{lem}
\begin{proof}
We define
\begin{equation}
A_{k}:=\{|u|>2^{k}\},\;k\in\mathbb{Z},
\end{equation}
and
\begin{equation}
D_{k}:=A_{k}\setminus A_{k+1}=\{2^{k}<f\leq2^{k+1}\}\,\,\, \text{and}\,\,\, d_{k}=|D_{k}|.
\end{equation}
Since $A_{k+1}\subseteq A_{k}$, it is easy to see
\begin{equation}\label{ak+1<ak}
a_{k+1}\leq a_{k}.
\end{equation}

By the assumption $u\in L^{\infty}(\mathbb{G})$ is compactly supported, $a_{k}$ and $d_{k}$ are bounded and vanish when $k$ is large enough. Also, we notice that the $D_{k}$'s are disjoint, therefore,
\begin{equation}\label{dl=cak}
\bigcup_{l\in\mathbb{Z},\,\,l\leq k}D_{l}=A^{c}_{k+1}
\end{equation}
and
\begin{equation}\label{dl=ak}
\bigcup_{l\in\mathbb{Z},\,\,l\geq k }D_{l}=A_{k}.
\end{equation}
From \eqref{dl=ak} we obtain that
\begin{equation}\label{sdl=ak}
\sum_{l\in\mathbb{Z},\,\,l\geq k }d_{l}=a_{k}
\end{equation}
and
\begin{equation}\label{dk=ak-sumdl}
d_{k}=a_{k}-\sum_{l\in\mathbb{Z},\,\,l\geq k+1 }d_{l}.
\end{equation}
Since  $a_{k}$ and $d_{k}$ are bounded and vanish when $k$ is large enough, \eqref{sdl=ak} and \eqref{dk=ak-sumdl} are convergent. We define the convergent series
\begin{equation}\label{sums}
S:=\sum_{\l\in\mathbb{Z},\,\,a_{l-1}\neq0}2^{lp}a_{l-1}^{-sp/Q}d_{l}.
\end{equation}
We have that $D_{k}\subseteq A_{k}\subseteq A_{k-1}$, therefore, $a_{i-1}^{-sp/Q}d_{l}\leq a_{i-1}^{-sp/Q}a_{l-1}$. Thus,
\begin{equation}\label{subseteq}
\{(i,l)\in\mathbb{Z}\,\,\text{s.t.}\,\,a_{i-1}\neq0\,\, \text{and}\,\,\,a^{-sp/Q}_{i-1}d_{l}\neq0 \}\subseteq\{(i,l)\in\mathbb{Z}\,\,\text{s.t.}\,\,a_{l-1}\neq0\}.
\end{equation}
By using \eqref{subseteq} and \eqref{ak+1<ak}, we calculate that

\begin{equation*}
\sum_{i\in\mathbb{Z},\,\,a_{i-1}\neq0}\sum_{l\in\mathbb{Z},\,\,l\geq i+1}2^{ip}a^{-sp/Q}_{i-1}d_{l}=\sum_{i\in\mathbb{Z},\,\,a_{i-1}\neq0}\sum_{l\in\mathbb{Z},\,\,l\geq i+1,\,\,a^{sp/Q}d_{l}\neq0}2^{ip}a^{-sp/Q}_{i-1}d_{l}
\end{equation*}

\begin{equation*}
\leq\sum_{i\in\mathbb{Z}}\sum_{l\in\mathbb{Z},\,\,l\geq i+1,\,\,a_{l-1}\neq0}2^{ip}a^{-sp/Q}_{i-1}d_{l}=\sum_{l\in\mathbb{Z},\,\,a_{l-1}\neq0}\sum_{i\in\mathbb{Z},\,\,i\leq l-1}2^{ip}a^{-sp/Q}_{i-1}d_{l}
\end{equation*}
\begin{equation}\label{14}
\leq\sum_{l\in\mathbb{Z},\,\,a_{l-1}\neq0}\sum_{i\in\mathbb{Z},\,\,i\leq l-1}2^{ip}a^{-sp/Q}_{l-1}d_{l}=\sum_{l\in\mathbb{Z},\,a_{l-1}\neq0}\sum^{+\infty}_{k=0}2^{p(l-1-k)}a_{l-1}^{-sp/Q}d_{l}\leq S.
\end{equation}

Notice that
\begin{equation*}
||u(x)|-|u(y)||\leq|u(x)-u(y)|,
\end{equation*}
for any $x,y \in \mathbb{G}.$
If we fix $i\in\mathbb{Z}$ and $x\in D_{i}$, then for any $j\in\mathbb{Z}$ with $j\leq i-2$, for any $y\in D_{j}$ using the above inequality, we obtain that
\begin{equation*}
|u(x)-u(y)|\geq2^{i}-2^{j+1}\geq2^{i}-2^{i-1}\geq2^{i-1}
\end{equation*}
and using \eqref{dl=cak}, we have
\begin{multline}\label{14,5}
\sum_{j\in\mathbb{Z},\,j\leq i-2}\int_{D_{j}}\frac{|u(x)-u(y)|^{p}}{q^{Q+sp}(y^{-1}\circ x)}dy\geq2^{(i-1)p}\sum_{j\in\mathbb{Z},\,j\leq i-2}\int_{D_{j}}\frac{dy}{q^{Q+sp}(y^{-1}\circ x)}\\
=2^{(i-1)p}\int_{A^{c}_{i-1}}\frac{dy}{q^{Q+sp}(y^{-1}\circ x)}.
\end{multline}
Now using \eqref{14,5} and Lemma \ref{lem1}, we obtain that
\begin{equation*}
\sum_{j\in\mathbb{Z},\,j\leq i-2}\int_{D_{j}}\frac{|u(x)-u(y)|^{p}}{q^{Q+sp}(y^{-1}\circ x)}dy\geq C 2^{ip}a_{i-1}^{-sp/Q},
\end{equation*}
with a positive constant $C$. That is,
for any $i\in\mathbb{Z}$, we have
\begin{equation}\label{15}
\sum_{j\in\mathbb{Z},\,\,j\leq i-2}\int_{D_{i}}\int_{D_{j}}\frac{|u(x)-u(y)|^{p}}{q^{Q+sp}(y^{-1}\circ x)}dxdy\geq C2^{ip}a^{-sp/Q}_{i-1}d_{i}.
\end{equation}
From \eqref{15} and \eqref{dk=ak-sumdl} we get
\begin{equation}\label{16}
\sum_{j\in\mathbb{Z},\,\,j\leq i-2}\int_{D_{i}}\int_{D_{j}}\frac{|u(x)-u(y)|^{p}}{q^{Q+sp}(y^{-1}\circ x)}dxdy\geq C\left(2^{ip}a^{-sp/Q}_{i-1}a_{i}-\sum_{l\in\mathbb{Z},\,l\geq i+1}2^{ip}a^{-sp/Q}_{i-1}d_{l}\right).
\end{equation}
By \eqref{15} and \eqref{sums} we establish that
\begin{equation}\label{17}
\sum_{i\in\mathbb{Z},\,a_{i-1}\neq0}\sum_{j\in\mathbb{Z},\,\,j\leq i-2}\int_{D_{i}}\int_{D_{j}}\frac{|u(x)-u(y)|^{p}}{q^{Q+sp}(y^{-1}\circ x)}dxdy\geq C \sum_{i\in\mathbb{Z},\,a_{i-1}\neq0}2^{ip}a^{-sp/Q}_{i-1}d_{i}\geq C\,S.
\end{equation}
Then, by using \eqref{14}, \eqref{16} and \eqref{17}, we obtain that
\begin{equation*}
\sum_{i\in\mathbb{Z},\,a_{i-1}\neq0}\sum_{j\in\mathbb{Z},\,\,j\leq i-2}\int_{D_{i}}\int_{D_{j}}\frac{|u(x)-u(y)|^{p}}{q^{Q+sp}(y^{-1}\circ x)}dxdy\geq C\sum_{i\in\mathbb{Z},\,a_{i-1}\neq0}2^{ip}a^{-sp/Q}_{i-1}a_{i}
\end{equation*}
\begin{equation*}
-C\sum_{i\in\mathbb{Z},\,a_{i-1}\neq0}\sum_{l\in\mathbb{Z},\,l\geq i+1}2^{ip}a^{-sp/Q}_{i-1}d_{l}\geq C\sum_{i\in\mathbb{Z},\,a_{i-1}\neq0}2^{ip}a^{-sp/Q}_{i-1}a_{i}-C\,S
\end{equation*}
\begin{equation*}
\geq C\sum_{i\in\mathbb{Z},\,a_{i-1}\neq0}2^{ip}a^{-sp/Q}_{i-1}a_{i}-\sum_{i\in\mathbb{Z},\,a_{i-1}\neq0}\sum_{j\in\mathbb{Z},\,\,j\leq i-2}\int_{D_{i}}\int_{D_{j}}\frac{|u(x)-u(y)|^{p}}{q^{Q+sp}(y^{-1}\circ x)}dxdy.
\end{equation*}

This means
\begin{equation}\label{18}
\sum_{i\in\mathbb{Z},\,a_{i-1}\neq0}\sum_{j\in\mathbb{Z},\,\,j\leq i-2}\int_{D_{i}}\int_{D_{j}}\frac{|u(x)-u(y)|^{p}}{q^{Q+sp}(y^{-1}\circ x)}dxdy\geq \frac{C}{2} \sum_{i\in\mathbb{Z},\,a_{i-1}\neq0}2^{ip}a^{-sp/Q}_{i-1}a_{i},
\end{equation}
for a constant $C>0$.
By symmetry and using \eqref{18}, we arrive at
\begin{align*}
[u]^{p}_{s,p,q}
 & =\int_{\mathbb{G}}\int_{\mathbb{G}}\frac{|u(x)-u(y)|^{p}}{q^{Q+sp}(y^{-1}\circ x)}dxdy=\sum_{i,j\in \mathbb{Z}}\int_{D_{i}}\int_{D_{j}}\frac{|u(x)-u(y)|^{p}}{q^{Q+sp}(y^{-1}\circ x)}dxdy
\\ &
\geq 2\sum_{i,j\in \mathbb{Z},\,j<i}\int_{D_{i}}\int_{D_{j}}\frac{|u(x)-u(y)|^{p}}{q^{Q+sp}(y^{-1}\circ x)}dxdy
 \\ &
\geq2\sum_{i\in\mathbb{Z},\,a_{i-1}\neq0}\sum_{j\in\mathbb{Z},\,\,j\leq i-2}\int_{D{i}}\int_{D_{j}}\frac{|u(x)-u(y)|^{p}}{q^{Q+sp}(y^{-1}\circ x)}dxdy
 \\ &
\geq  C \sum_{i\in\mathbb{Z},\,a_{i-1}\neq0}2^{ip}a^{-sp/Q}_{i-1}a_{i}.
\end{align*}
Lemma \ref{3.3} is proved.
\end{proof}

\begin{lem}\label{3.4}
Let $1< p<\infty$ and $u:\mathbb{G}\rightarrow\mathbb{R}$ be a measurable function. For any $n\in\mathbb{R}$
\begin{equation}
u_{n}:=\max\{\min\{u(x),n\},-n\},\,\,\,for\,\, any\,\,x\in \mathbb{G}.
\end{equation}
Then,
$$\lim_{n\rightarrow +\infty}\|u_{n}\|_{L^{p}(\mathbb{G})}=\|u \|_{L^{p}(\mathbb{G})}.$$
\end{lem}
\begin{proof}
The proof is the same as in	\cite[Lemma 6.4]{DPV}.
	\end{proof}

By using the above lemmas we prove the following analogue of the fractional Sobolev inequality on $\mathbb{G}$:
\begin{thm}\label{sob}
Let $p>1$, $s\in(0,1)$, $Q>sp$ and $q$ be a quasi-norm on $\mathbb{G}$. For any measurable and compactly supported function $u:\mathbb{G}\rightarrow \mathbb{R}$ there exists a positive constant $C=C(Q,p,s,q)>0$ such that
  \begin{equation}\label{sobin1}
  ||u||^{p}_{L^{p^{*}}(\mathbb{G})}\leq C[u]^{p}_{s,p,q},
  \end{equation}
  where $p^{*}=p^{*}(Q,s)=\frac{Qp}{Q-sp}$.
\end{thm}
\begin{proof}
First of all, we suppose that Gagliardo's seminorm $[u]_{s,p,q}$ is bounded, i.e.
\begin{equation}\label{22}
[u]^{p}_{s,p,q}=\int_{\mathbb{G}}\int_{\mathbb{G}}\frac{|u(x)-u(y)|^{p}}{q^{Q+sp}(y^{-1}\circ x)}dxdy<+\infty.
\end{equation}
and we suppose that $u\in L^{\infty}(\mathbb{G})$.

If \eqref{22} is satisfied for bounded functions, it holds also for the function $u_{n}$, obtained by $u$ cutting at levels $-n$ and $n$. Thus, by using Lemma \ref{3.4} and \eqref{22} with the dominated convergence theorem, we obtain that
\begin{equation*}
\lim_{n\rightarrow+\infty}[u_{n}]^{p}_{s,p,q}=\lim_{n\rightarrow+\infty}
\int_{\mathbb{G}}\int_{\mathbb{G}}\frac{|u_{n}(x)-u_{n}(y)|^{p}}{q^{Q+sp}(y^{-1}\circ x)}dxdy
\end{equation*}
\begin{equation}\label{23}
=\int_{\mathbb{G}}\int_{\mathbb{G}}\frac{|u(x)-u(y)|^{p}}{q^{Q+sp}(y^{-1}\circ x)}dxdy=[u]^{p}_{s,p,q}.
\end{equation}
As in Lemma \ref{3.3} we define $a_{k}$ and $A_{k}$, so we have
\begin{equation*}
||u \|_{L^{p^{*}}(\mathbb{G})}=\left(\sum_{k\in\mathbb{Z}}\int_{A_{k}\setminus A_{k+1}}|u(x)|^{p^{*}}dx\right)^{1/p^{*}}\leq\left(\sum_{k\in\mathbb{Z}}\int_{A_{k}\setminus A_{k+1}}2^{(k+1)p^{*}}dx\right)^{1/p^{*}}
\end{equation*}
\begin{equation}
\leq\left(\sum_{k\in\mathbb{Z}}2^{(k+1)p^{*}}a_{k}\right)^{1/p^{*}}.
\end{equation}
Then, with $p/p^{*}=1-sp/Q<1$ and $T=2^{p}$, Lemma \ref{sob2} yields
\begin{multline}
\|u\|^{p}_{L^{p^{*}}(\mathbb{G})}\leq2^{p}
\left(\sum_{k\in\mathbb{Z}}2^{kp^{*}}a_{k}
\right)^{p/p^{*}}\leq2^{p}\sum_{k\in\mathbb{Z}}2^{kp}a^{(Q-sp)/Q}_{k}\\ \leq C\sum_{k\in\mathbb{Z},\,a_{k}\neq0}2^{kp}a^{-sp/Q}_{k}a_{k+1}
\end{multline}
for a positive constant $C=C(Q,p,s,q)>0$.

Finally, using Lemma \ref{3.3} we arrive at
\begin{equation}
 \|u\|^{p}_{L^{p^{*}}(\mathbb{G})}\leq C\sum_{k\in\mathbb{Z},\,a_{k}\neq0}2^{kp}a^{-sp/Q}_{k}a_{k+1} \leq C\int_{\mathbb{G}}\int_{\mathbb{G}}\frac{|u(x)-u(y)|^{p}}{q^{Q+sp}(y^{-1}\circ x)}dxdy= C[u]^{p}_{s,p,q}.
\end{equation}

Theorem \ref{sob} is proved.
\end{proof}

Now to prove an analogue of the fractional Hardy inequality we need some preliminary results.
\begin{lem}[{\cite{FS}}, Lemma 2.6]\label{nineq}
Assume that $p>1$, then for all $t\in[0,1]$ and $a\in\mathbb{C}$, we have
\begin{equation}
|a-t|^{p}\geq(1-t)^{p-1}(|a|^{p}-t).
\end{equation}
\end{lem}

\begin{lem}(Picone-type inequality)\label{Picone}
Let $\omega\in W_{0}^{s,p,q}(\Omega)$ be $w>0$ in $\Omega\subset \mathbb{G}$. Assume that $(-\Delta_{p,q})^{s}\omega=\nu>0$ with $\nu\in L^{1}_{loc}(\Omega)$ , then for all $u\in C_{0}^{\infty}(\Omega)$, we have
\begin{equation}
\frac{1}{2}\int_{\Omega}\int_{\Omega}\frac{|u(x)-u(y)|^{p}}{q^{Q+ps}(y^{-1}\circ x)}dxdy\geq\left<(-\Delta_{p,q})^{s}\omega,\frac{|u|^{p}}{\omega^{p-1}}\right>.
\end{equation}

\end{lem}
\begin{proof}
We set $v=\frac{|u|^{p}}{|\omega|^{p-1}}$ and $k(x,y)=\frac{1}{q^{Q+ps}(y^{-1}\circ x)}$, then we have
\begin{equation*}
\left<(-\Delta_{p,q})^{s}\omega,v\right>=\int_{\Omega}v(x)dx\int_{\Omega}|\omega(x)-\omega(y)|^{p-2}(\omega(x)-\omega(y))k(x,y)dy
\end{equation*}
\begin{equation*}
=\int_{\Omega}\frac{|u|^{p}}{|\omega|^{p-1}}dx\int_{\Omega}|\omega(x)-\omega(y)|^{p-2}(\omega(x)-\omega(y))k(x,y)dy.
\end{equation*}

Let us prove that $k(x,y)$ is symmetric, i.e. $k(x,y)=k(y,x)$ for all $x,y\in \mathbb{G}$.  By the definition of quasi-norm we have $q(x^{-1})=q(x)$ for all $x\in\mathbb{G}$. So, by using this fact we obtain
\begin{equation*}
k(x,y)=\frac{1}{q^{Q+ps}(y^{-1}\circ x)}=\frac{1}{q^{Q+ps}(z)}=\frac{1}{q^{Q+ps}(z^{-1})}
\end{equation*}
\begin{equation*}
=\frac{1}{q^{Q+ps}((y^{-1}\circ x)^{-1})}=\frac{1}{q^{Q+ps}(x^{-1}\circ y)}=k(y,x),
\end{equation*}
for all $x,y\in \mathbb{G}$.
Now since $k(x,y)$ is symmetric, we establish that
\begin{equation*}
\left<(-\Delta_{p,q})^{s}\omega,v\right>=
\end{equation*}
\begin{equation*}
\frac{1}{2}\int_{\Omega}\int_{\Omega}\left(\frac{|u(x)|^{p}}{|\omega(x)|^{p-1}}-\frac{|u(y)|^{p}}{|\omega(y)|^{p-1}}\right)|\omega(x)-\omega(y)|^{p-2}(\omega(x)-\omega(y))k(x,y)dydx.
\end{equation*}

Let $g=\frac{u}{\omega}$ and $$R(x,y)=|u(x)-u(y)|^{p}-(|g(x)|^{p}\omega(x)-|g(y)|^{p}\omega(y))|\omega(x)-\omega(y)|^{p-2}(\omega(x)-\omega(y)),$$ then we have

\begin{equation*}
\left<(-\Delta_{p,q})^{s}\omega,v\right>+\frac{1}{2}\int_{\Omega}\int_{\Omega}R(x,y)k(x,y)dydx=\frac{1}{2}\int_{\Omega}\int_{\Omega}|u(x)-u(y)|^{p}k(x,y)dydx.
\end{equation*}
 By the symmetry argument, we can assume that $\omega(x)\geq\omega(y)$. By using Lemma \ref{nineq} with $t=\frac{\omega(y)}{\omega(x)}$ and $a=\frac{g(x)}{g(y)}$ and we establish that $R(x,y)\geq0$. Thus, we have proved the inequality

\begin{equation*}
\left<(-\Delta_{p,q})^{s}\omega,v\right>\leq\frac{1}{2}\int_{\Omega}\int_{\Omega}\frac{|u(x)-u(y)|^{p}}{q^{Q+sp}(y^{-1}\circ x)}dydx.
\end{equation*}
Lemma \ref{Picone} is proved.
\end{proof}

\begin{lem}\label{sh}
Let $\omega=q^{-\gamma}(x)$ with $\gamma\in\left( 0,\frac{Q-ps}{p-1}\right)$, then there exists a positive constant $\mu(\gamma)>0$ such that
\begin{equation}
(-\Delta_{p,q})^{s}(q^{-\gamma}(x))=\mu(\gamma)\frac{1}{q^{ps+\gamma(p-1)}(x)}\,\,
\text{a.e.}\; \text{in}\;\mathbb{G}\setminus\{0\}.
\end{equation}
\end{lem}
\begin{proof}
We set $r=q(x)$ and $\rho=q(y)$ with $x=r x'$ and $y=\rho y'$ where $q(x')=q(y')=1$. Then, we have
\begin{equation*}
(-\Delta_{p,q})^{s}\omega
\end{equation*}
\begin{equation*}
=\int^{+\infty}_{0}|q^{-\gamma}(x)-q^{-\gamma}(y)|^{p-2}(q^{-\gamma}(x)-q^{-\gamma}(y))q^{Q-1}(y)\left(\int_{q(y')=1}\frac{d\sigma(y)}{q^{Q+ps}(y^{-1}\circ x)}\right)dq(y)
\end{equation*}
\begin{equation*}
=\frac{1}{q^{ps+\gamma(p-1)}(x)}\int^{+\infty}_{0}\left|1-\frac{q^{-\gamma}(y)}{q^{-\gamma}(x)}\right|^{p-2}\times
\end{equation*}
\begin{equation*}
\times\left(1-\frac{q^{-\gamma}(y)}{q^{-\gamma}(x)}\right)\frac{q^{Q-1}(y)}{q^{Q-1}(x)}\left(\int_{q(y')=1}\frac{d\sigma(y)}{q^{Q+ps}\left(\left(\frac{q(y)}{q(x)}y'\right)^{-1}\circ x'\right)}\right)dq(y).
\end{equation*}
Let $\rho=\frac{q(y)}{q(x)}$ and $L(\rho)=\int_{q(y')=1}\frac{d\sigma(y)}{q^{Q+ps}((\rho y')^{-1}\circ x')}$, we have
\begin{equation*}
(-\Delta_{p,q})^{s}\omega=\frac{1}{q^{ps+\gamma(p-1)}(x)}\int^{+\infty}_{0}|1-\rho^{-\gamma}|^{p-2}(1-\rho^{-\gamma})L(\rho)\rho^{Q-1}d\rho.
\end{equation*}
It easy to see
\begin{equation}\label{mu}
\mu(\gamma)=\int^{+\infty}_{0}\phi(\rho)d\rho
\end{equation}
with
$\phi(\rho)=|1-\rho^{-\gamma}|^{p-2}(1-\rho^{-\gamma})L(\rho)\rho^{Q-1}$.

Now it remains to show that $\mu(\gamma)$ is a positive and bounded. Firstly, let us show that $\mu(\gamma)$ is bounded. We have
\begin{equation}
\mu(\gamma)=\int^{1}_{0}\phi(\rho)d\rho+\int^{+\infty}_{1}\phi(\rho)d\rho=I_{1}+I_{2}.
\end{equation}
Using the new variable $\zeta=\frac{1}{\rho}$ we have $L(\rho)=L\left(\frac{1}{\zeta}\right)=\zeta^{Q+ps}L(\zeta)$ for any $\zeta>0$. Thus, we establish
\begin{equation}\label{asymp}
\mu(\gamma)=\int^{+\infty}_{1}(\rho^{-\gamma}-1)^{p-1}(\rho^{Q-1-\gamma(p-1)}-\rho^{ps-1})L(\rho)d\rho.
\end{equation}
For $\rho\rightarrow 1$ we have
 \begin{equation}
(\rho^{-\gamma}-1)^{p-1}(\rho^{Q-1-\gamma(p-1)}-\rho^{ps-1})L(\rho)\simeq(\rho-1)^{-1-ps+p}\in L^{1}(1,2).
\end{equation}
Similarly, for $\rho\rightarrow \infty$ we get
 \begin{equation}
(\rho^{-\gamma}-1)^{p-1}(\rho^{Q-1-\gamma(p-1)}-\rho^{ps-1})L(\rho)\simeq\rho^{-1-ps}\in L^{1}(2,\infty).
\end{equation}
These show that $\mu(\gamma)$ is bounded. On the other hand, by \eqref{asymp} with $\gamma\in\left( 0,\frac{Q-ps}{p-1}\right)$ we see that $\mu(\gamma)$ is positive.

Lemma \ref{sh} is proved.
\end{proof}

As a result we establish the following analogue of the fractional Hardy inequality on $\mathbb{G}$.

\begin{thm}\label{harun1}
For all $u\in C^{\infty}_{0}(\mathbb{G})$ we have
\begin{equation}\label{harun}
2\mu(\gamma)\int_{\mathbb{G}}\frac{|u(x)|^{p}}{q^{ps}(x)}dx\leq[u]^{p}_{s,p,q},
\end{equation}
where $p\in(1,\infty),\,s\in(0,1)$ and $C$ is positive constant.
\end{thm}
\begin{proof}
Let $u\in C^{\infty}_{0}(\mathbb{G})$ and $\gamma<\frac{Q-ps}{p-1}$. By Lemma \ref{sh} and Lemma  \ref{Picone} we establish that
\begin{equation*}
\frac{1}{2}[u]^{p}_{s,p,q}=\frac{1}{2}\int_{\mathbb{G}}\int_{\mathbb{G}}\frac{|u(x)-u(y)|^{p}}{q^{Q+ps}(y^{-1}\circ x)}dxdy\geq\left<(-\Delta_{p,q})^{s}(q^{-\gamma}(x)),\frac{|u(x)|^{p}}{q^{-\gamma(p-1)}(x)}\right>
\end{equation*}
\begin{equation}
=\mu(\gamma)\int_{\mathbb{G}}\frac{|u(x)|^{p}}{q^{ps}(x)}dx.
\end{equation}
It completes the proof of Theorem \ref{harun1}.
\end{proof}

\section{Lyapunov-type Inequality}
\label{SEC:3}
In this section we prove a Lyapunov-type inequality for the fractional $p$-sub-Laplacian with a homogeneous Dirichlet boundary problem on  $\mathbb{G}$.  Let $p>1$ and $s\in(0,1)$ be such that $Q>sp$ and $\Omega\subset\mathbb{G}$ be a Haar measurable set. We denote by $r_{\Omega,q}$ the inner quasi-radius of $\Omega$, that is,
\begin{equation}
r_{\Omega,q}=\max\{q(x):\,\,x\in\Omega\}.
\end{equation}
Let us consider
 \begin{equation}\label{fplapg}
 \begin{cases}
   (-\Delta_{p,q})^{s}u(x)=\omega|u(x)|^{p-2}u(x),\,\,x\in\Omega,\\
   u(x)=0,\,\,\,\,x\in\mathbb{G}\setminus\Omega,
 \end{cases}
\end{equation}
 where $\omega \in L^{\infty}(\Omega)$.
A function $u\in W_{0}^{s,p,q}(\Omega)$ is called a weak solution of the problem \eqref{fplapg} if
\begin{equation}
\int_{\Omega}\int_{\Omega}\frac{|u(x)-u(y)|^{p-2}(u(x)-u(y))(v(x)-v(y))}{q^{Q+sp}(y^{-1}\circ x)}dxdy=\int_{\Omega}\omega(x)|u(x)|^{p-2}u(x)v(x)dx
\end{equation}
for all $v\in W_{0}^{s,p,q}(\Omega)$,
 \begin{thm}\label{thmlyap} Let $\Omega\subset \mathbb{G}$ be a Haar measurable set.
 Let $\omega\in L^{\theta}(\Omega)$ be a non-negative weight with $\frac{Q}{sp}<\theta<\infty$. Suppose that problem \eqref{fplapg} with $Q>ps$ has a non-trivial weak solution $u\in W_{0}^{s,p,q}(\Omega)$. Then, we have
 \begin{equation}\label{lyapineq}
 \|\omega\|_{L^{\theta}(\Omega)}\geq\frac{C}{r_{\Omega,q}^{sp-Q/\theta}},
 \end{equation}
 where $C=C(Q,p,s,q)>0$.
 \end{thm}

 \begin{proof}
Let us define
$$\beta=\alpha p+(1-\alpha)p^{*},$$
where $\alpha=\frac{\theta-\theta/sp}{\theta-1}\in(0,1)$ and  $p^{*}$ is the Sobolev conjugate exponent as in Theorem \ref{sob}. Let $\beta=p\theta'$ with $1/\theta+1/\theta'=1$. Then, we have
\begin{equation}\label{inn}
\int_{\Omega}\frac{|u(x)|^{\beta}}{r^{\alpha sp}_{\Omega,q}}dx\leq\int_{\Omega}\frac{|u(x)|^{\beta}}{q^{\alpha sp}(x)}dx.
\end{equation}
Now, H\"{o}lder's inequality with exponents $\nu=\alpha^{-1}$ and $1/\nu+1/\nu'=1$ gives
\begin{equation}
\int_{\Omega}\frac{|u(x)|^{\beta}}{q^{\alpha sp}(x)}dx\leq\int_{\Omega}\frac{|u(x)|^{\alpha p}|u(x)|^{(1-\alpha) p^{*}}}{q^{\alpha sp}(x)}dx\leq\left(\int_{\Omega}\frac{|u(x)|^{p}}{q^{ sp}(x)dx}\right)^{\alpha}\left(\int_{\Omega}|u(x)|^{p^{*}}dx\right)^{1-\alpha}.
\end{equation}
Then, by using Theorem \ref{sob} and \ref{harun1}, we obtain that
\begin{equation*}
\int_{\Omega}\frac{|u(x)|^{\beta}}{q^{\alpha sp}(x)}dx\leq C^{\alpha}_{1}\left(\int_{\Omega}\int_{\Omega}\frac{|u(x)-u(y)|^{p}}{q^{Q+sp}(y^{-1}\circ x)}dxdy\right)^{\alpha/p}\,\, C^{(1-\alpha)p^{*}/p}_{2}[u]_{s,p,q}^{(1-\alpha)p^{*}/p}
\end{equation*}
\begin{equation*}
\leq C^{\alpha}_{1}[u]_{s,p,q}^{\alpha}C^{(1-\alpha)p^{*}/p}_{2}[u]_{s,p,q}^{(1-\alpha)p^{*}/p}=C \left([u]_{s,p,q}^{p}\right)^{(\alpha p+(1-\alpha)p^{*})/p}=C\left(\int_{\Omega}\omega(x)|u(x)|^{p}dx\right)^{\theta'}
\end{equation*}
\begin{equation*}
\leq C\left(\int_{\Omega}\omega^{\theta}(x)dx\right)^{\theta'/\theta}\int_{\Omega}|u(x)|^{p\theta'}dx=C\|\omega\|^{\theta'}_{L^{\theta}(\Omega)}\int_{\Omega}|u(x)|^{\beta}dx.
\end{equation*}
That is, we have
\begin{equation*}
\int_{\Omega}\frac{|u(x)|^{\beta}}{q^{\alpha sp}(x)}dx\leq C\|\omega\|^{\theta'}_{L^{\theta}(\Omega)}\int_{\Omega}|u(x)|^{\beta}dx.
\end{equation*}
Thus, from \eqref{inn} we get
\begin{equation}
\frac{1}{r^{\alpha sp}_{\Omega,q}}\int_{\Omega}|u(x)|^{\beta}dx\leq \int_{\Omega}\frac{|u(x)|^{\beta}}{q^{\alpha sp}(x)}dx\leq C\|\omega\|^{\theta'}_{L^{\theta}(\Omega)}\int_{\Omega}|u(x)|^{\beta}dx.
\end{equation}
Finally, we arrive at
\begin{equation}
\frac{C}{r_{\Omega,q}^{sp-Q/\theta}}\leq \|\omega \|_{L^{\theta}(\Omega)}.
\end{equation}
Theorem \ref{thmlyap} is proved.
 \end{proof}

 Let consider the following spectral problem for the non-linear, fractional $p$-sub-Laplacian $(-\Delta_{p,q})^{s},\,\,\,1<p<\infty,\,\,s\in(0,1),$ with Dirichlet boundary condition:
 \begin{equation}\label{spl}
 \begin{cases}
   (-\Delta_{p,q})^{s}u=\lambda|u|^{p-2}u,\,\,x\in\Omega,\\
   u(x)=0,\,\,x\in \mathbb{G}\setminus\Omega.
 \end{cases}
\end{equation}
We have the following Rayleigh quotient for the fractional Dirichlet $p$-sub-Laplacian (cf. \cite{Kir})

\begin{equation}\label{rq}
 \lambda_{1}=\inf_{u\in W_{0}^{s,p,q}(\Omega),\,\,u\neq0}\frac{[u]^{p}_{s,p,q}}{\|u\|^{p}_{L^{p}(\mathbb{G})}}.
\end{equation}
As a consequence of Theorem \ref{thmlyap} we obtain
\begin{thm}\label{sfplap}
Let $\lambda_{1}$ be the first eigenvalue of  problem \eqref{spl} given by \eqref{rq}. Let $Q>sp,\,\,s\in(0,1)$ and $1<p<\infty.$ Then we have
\begin{equation}
\lambda_{1}\geq \sup_{\frac{Q}{sp}<\theta<\infty}\frac{C}{|\Omega|^{\frac{1}{\theta}}r_{\Omega,q}^{sp-Q/{\theta}}},
\end{equation}
where $C$ is a positive constant given in Theorem \ref{thmlyap}, $|\cdot|$ is the Haar measure and $r_{\Omega,q}$ is the inner quasi-radius of  $\Omega$.
\end{thm}
\begin{proof}
In Theorem \ref{thmlyap} taking $\omega=\lambda\in L^{\theta}(\Omega)$ and using Lyapunov-type inequality \eqref{lyapineq}, we get that
\begin{equation}
\|\omega \|_{L^{\theta}(\Omega)}=\|\lambda \|_{L^{\theta}(\Omega)}=\left(\int_{\Omega}\lambda^{\theta} dx\right)^{1/\theta}\geq\frac{C}{r_{\Omega,q}^{sp-Q/\theta}}.
\end{equation}
For every $\theta>\frac{Q}{sp}$, we have
\begin{equation}
\lambda_{1}\geq \frac{C}{|\Omega|^{\frac{1}{\theta}}r_{\Omega,q}^{sp-Q/{\theta}}}.
\end{equation}
Thus, we establish
\begin{equation}
\lambda_{1}\geq \sup_{\frac{Q}{sp}<\theta<\infty}\frac{C}{|\Omega|^{\frac{1}{\theta}}r_{\Omega,q}^{sp-Q/{\theta}}},
\end{equation}
 for all $\frac{Q}{sp}<\theta<\infty$.
Theorem \ref{sfplap} is proved.
\end{proof}

\end{document}